\documentclass[12pt, reqno]{amsart}
\usepackage{amsmath}
\usepackage{amsthm, color}
\usepackage{amsfonts}
\usepackage{amssymb}
\usepackage{mathrsfs}
\usepackage{bbm}
\usepackage[margin=1in]{geometry}
\usepackage{setspace}
\usepackage{tikz}
\usetikzlibrary{matrix, arrows}
\usepackage{tikz-cd}
\usepackage{enumitem}
\usepackage[colorlinks, breaklinks]{hyperref}
\hypersetup{linkcolor=black,citecolor=blue,filecolor=black,urlcolor=black} 
\usepackage{pdflscape}
\usepackage{array}
\usepackage{adjustbox}
\usepackage{tabularx}
\usepackage{multirow}
\usepackage{multicol}
\usepackage{arydshln}
\usepackage{todonotes}

\newcommand{\abs}[1]{\left\vert #1 \right\vert}

\newcommand{\ceiling}[1]{\left\lceil #1 \right\rceil}

\DeclareMathOperator{\codim}{codim}
\DeclareMathOperator{\cx}{cx}  			

\DeclareMathOperator{\depth}{depth}         
\renewcommand{\emptyset}{\varnothing}
\renewcommand{\epsilon}{\varepsilon}
\newcommand{\floor}[1]{\left\lfloor #1 \right\rfloor}

\DeclareMathOperator{\hgt}{ht}  			

\DeclareMathOperator{\im}{im}

\newcommand{\iso}{\cong}

\newcommand{\kk}{\mathbbm{k}}		    
\newcommand{\longto}{\longrightarrow}

\DeclareMathOperator{\pd}{pd}  			
\renewcommand{\phi}{\varphi}

\DeclareMathOperator{\reg}{reg}			

\renewcommand{\setminus}{\smallsetminus}

\newcommand{\term}[1]{\textbf{#1}}	
\newcommand{\tensor}{\otimes}

\renewcommand{\vec}[1]{\mathbf{#1}}
\newcommand{\ZZ}{\mathbb{Z}}

\newcommand{\graph}[3][]{
	\begin{tikzpicture}[#1]         
	\newcommand*\points{#2}     
	\newcommand*\edges{#3}          
	\newcommand*\scale{0.75}          
	\foreach \x/\y/\z/\w in \points {
		\draw[fill = black!50] (\scale*\x,\scale*\y) circle [radius = 0.1] node[label = {[label distance = 0.05 cm]\w: $\z$}] (\z) {}; 
	}
	\foreach \x/\y in \edges { \draw (\x) -- (\y); }      
	\end{tikzpicture}
}

\newcommand{\labeledgraph}[3][]{
	\begin{tikzpicture}[#1]    
	\newcommand*\points{#2}     
	\newcommand*\edges{#3}          
	\newcommand*\scale{0.75}          
	\foreach \x/\y/\a/\z/\w in \points {
		\draw[fill = black!50] (\scale*\x,\scale*\y) circle [radius = 0.1] node[label = {[label distance = 0.05 cm]\w: $\z$}] (\a) {}; 
	}
	\foreach \x/\y in \edges { \draw (\x) -- (\y); }      
	\end{tikzpicture}
}

\newtheorem{thm}{Theorem}[section]

\newtheorem{prop}[thm]{Proposition}
\newtheorem{cor}[thm]{Corollary}

\newtheorem*{main-thm}{Main Theorem}
\newtheorem*{thm1}{Theorem~\ref{general:bound}}
\newtheorem*{thm2}{Theorem~\ref{bipartitebound}}

\theoremstyle{definition}

\newtheorem{example}[thm]{Example}
\newtheorem{rmk}[thm]{Remark}

\newtheorem*{notation}{Notation}

\numberwithin{equation}{section}
\numberwithin{table}{section}
\numberwithin{figure}{section}

\title{Depth and Singular Varieties of Exterior Edge Ideals}
\author[M. Mastroeni]{Matthew Mastroeni}
\address{Iowa State University, Department of Mathematics, Ames, IA, USA}
\email{mmastro@iastate.edu}
\author[J. McCullough]{Jason McCullough}
\address{Iowa State University, Department of Mathematics, Ames, IA, USA}
\email{jmccullo@iastate.edu}
\author[A. Osborne]{Andrew Osborne}
\address{Iowa State University, Department of Mathematics, Ames, IA, USA}
\email{osborne1@iastate.edu}
\author[J. Rice]{Joshua Rice}
\address{Iowa State University, Department of Mathematics, Ames, IA, USA}
\email{jar238@iastate.edu}
\author[C. Willis]{Cole Willis}
\address{University of Minnesota, School of Mathematics, Minneapolis, MN, USA}
\email{will7405@umn.edu}
\date{}

\begin{document}



\begin{abstract} Edge ideals of finite simple graphs are well-studied over polynomial rings.  In this paper, we initiate the study of edge ideals over exterior algebras, specifically focusing on the depth and singular varieties of such ideals.  We prove an upper bound on the depth of the edge ideal associated to a general graph and a more refined bound for bipartite graphs, and we show that both are tight.  We also compute the depth of several large families of graphs including cycles, complete multipartite graphs, spider graphs, and Ferrers graphs.  Finally, we focus on the effect whiskering a graph has on the depth of the associated edge ideal.
\end{abstract}
\maketitle

\begin{spacing}{1.2}
\section{Introduction}

Let $V$ be a vector space with basis $e_1,\ldots,e_n$ over a field $\mathbbm{k},$ and let $E = \bigwedge_\mathbbm{k} \langle e_1,\ldots,e_n \rangle$ be the exterior algebra of $V$. The standard basis elements $e_{k_1} \wedge \cdots \wedge e_{k_s}$ of $E,$ with $k_1 < \ldots < k_s,$ are called monomials in $E,$ and an ideal $I$ of $E$ generated  by monomials is called a monomial ideal. Given a finite, simple graph $G$ with vertex set $v_1,\ldots,v_n$, we consider the corresponding edge ideal $I_E(G) = (e_i\wedge e_j \mid \{v_i,v_j\} \text{ is an edge of } G)$.  The corresponding ideals $I_S(G)$ over a polynomials ring $S$ are well-studied, where there are many results linking combinatorial properties of $G$ to algebraic properties of $I_S(G)$.  The situation over exterior algebras is the focus of this paper.

Monomial ideals more generally were studied by Aramova, Avramov, and Herzog.  It follows from \cite[1.3]{resolutions:of:monomial:ideals:over:exterior:algebras} that $I_S(G)$ and $I_E(G)$ have the same regularity for any graph $G$.  On the other hand, the depths of $E/I_E(G)$ and $S/I_S(G)$ can be quite different.  The aim of this paper is to bound, and in some cases compute exactly, the depth of $E/I_E(G)$ in terms of combinatorial properties of $G$.

Our main results include a tight general upper bound on depth of edge ideals.  

\begin{thm1}
Let $G$ be a graph on $n$ vertices, none of which are isolated.  Then
\[ \depth_S G \le n + 2 - 2\sqrt{n} \qquad\text{and}\qquad \depth_E G \le n + 1 - 2\sqrt{n}.\]
Moreover, when $n$ is a square, both sets of bounds are tight.
\end{thm1}
\noindent The first inequality recovers a result of H\`{a} and Hibi \cite[4.2]{HH21}.  
In the case of bipartite graphs, we prove a similar upper bound on the  depth of edge ideals.

\begin{thm2}
Let $G$ be a bipartite graph on $n$ vertices, none of which are isolated.  Then
\[ \depth_S G \le \left\lfloor \frac{n}{2} \right \rfloor \qquad\text{and}\qquad 
    \depth_E G \le \left\lfloor \frac{n}{2} \right \rfloor - 1.\]
Moreover, both bounds are tight.
\end{thm2}

We also compute the depths and singular varieties of exterior edge ideals of several large classes of graphs including complete graphs, complete multipartite graphs, cycles, paths, spiders, and star graphs.  We show how one can algorithmically compute the depth of any tree, although an exact formula is likely not possible; see Example~\ref{depth:of:tree}.  Finally, we investigate the effect of whiskering a graph on the depth of the associated edge ideal, which is distinct from the situation over a polynomial ring; see Corollary~\ref{depth:second:whiskering} and Remark~\ref{whisker:remark}.

The rest of this paper is organized as follows. Section~\ref{section:background} collects the necessary background and notation that we need for the remainder of the paper and includes several propositions that will be useful in our later proofs. In Section~\ref{section:singularvariety}, we compute the depth and singular varieties of some general families of graphs, such as cycle graphs, complete multipartite graphs, and whiskered graphs. In Section~\ref{section:boundsOnDepth}, we provide several upper bounds on the depth of edge ideals. In particular, we establish tight upper bounds over $E$ and $S$ of the depth of an edge ideal and a tight upper bound of the projective dimension of an edge ideal over $S$ and a tight upper bound of the complexity of an edge ideal over $E.$

\section{Background}\label{section:background}

\subsection{Depth and Singular Varieties}

Let $E = \bigwedge_\kk \langle e_1, \dots, e_n \rangle$ denote an exterior algebra over the field $\kk$ in $n$ variables.  We recall that $E$ has a $\kk$-basis consisting of all monomials $e_T = e_{i_1}e_{i_2} \cdots e_{i_r}$ for each subset $T = \{i_1 < i_2 < \cdots < i_r\} \subseteq [n] = \{1, 2, \dots, n\}$.  (We omit the wedge products in the remainder of the paper.)  There is a natural grading $E = \bigoplus_{r = 0}^n E_r$, where $E_r$ is the subspace spanned by all monomials $e_T$ with $\abs{T} = r$, and we call elements of $E_r$ homogeneous of degree $r$.  See \cite{Herzog:Hibi:monomial:ideals} for further details and any unexplained terminology.

Let $M$ be a graded $E$-module, and take $\ell \in E_1$. We say that $\ell$ is \term{$M$-regular} if 
\[ (0 :_M \ell) := \{ m\in M \mid \ell m = 0 \} = \ell M.\]
Otherwise $\ell$ is \term{$M$-singular}. The set of all $M$-singular elements in $E_1$ is called the \term{singular variety} of $M$, which we  denote by $V_E(M)$.  (This is called the rank variety of $M$ in  {\cite{resolutions:of:monomial:ideals:over:exterior:algebras}}.)  A sequence $\ell_1, \dots, \ell_r$ of linear forms in $E$ is called an \term{$M$-regular sequence} if $\ell_i$ is regular on $M/(\ell_1, \dots, \ell_{i-1})M$ for all $i \leq r$. The \term{depth} of $M$, denoted $\depth_E M$, is the maximum length of such a regular sequence. The following result gathers useful results about the singular varieties of modules over $E$.

\begin{thm}[{\cite[3.1, 3.2, 4.1 ]{resolutions:of:monomial:ideals:over:exterior:algebras}}] \label{singular:varieties}
Let $E$ be an exterior algebra in $n$ variables over an algebraically closed field $\kk$, and let $L$, $M$, and $N$ be finitely generated graded $E$-modules.  Then:
\begin{enumerate}[label = \textnormal{(\alph*)}]
    \item $V_E(M)$ is a cone in the vector space $E_1$.
    \item If $I \subseteq E$ is a monomial ideal, then $V_E(E/I)$ is a union of finitely many coordinate subspaces of $E_1$.
    \item $\depth_E M = \codim_{E_1} V_E(M)$
    \item If $0 \to L \to M \to N \to 0$ is an exact sequence of $E$-linear maps, then any one of the sets $V_E(L)$, $V_E(M)$, $V_E(N)$ is contained in the union of the other two.
    \item $V_E(M \tensor_\kk N) = V_E(M) \cap V_E(N)$
\end{enumerate}
\end{thm}

As an easy corollary, we prove a depth lemma for modules in a short exact sequence.

\begin{cor} \label{depth:sequence:inequality}
Let $E$ be an exterior algebra over an algebraically closed field $\kk$, and let $0 \to L \to M \to N \to 0$ be an exact sequence of finitely generated graded $E$-modules.  Then the depth of any one of the modules $L$, $M$, $N$ is at least the minimum of the depths of the other two, and equality holds if the other two modules have different depths.
\end{cor}

\begin{proof}We only prove the corollary statement for $L$ since the remaining arguments are identical. Suppose without loss of generality that $\depth_E M \leq \depth_E N.$ By part (c) of the preceding theorem
\[\dim V_E(N) = n - \depth_E N  \leq n - \depth_E M = \dim V_E(M).\] Part (d) of the preceding theorem yields 
\[\dim V_E (L) \leq \max\{ \dim V_E (M), \dim V_E (N) \} =\dim V_E(M).\] 
Therefore
\[ \depth_E M = n - \dim V_E(M) \leq  n - \dim V_E(L) = \depth_E L,\]
 proving the first claim. 

We now prove the second claim. Suppose that $\depth_E M < \depth_E N.$ Using the previous claim, we have the two inequalities 
\[\depth_E L \geq \min\{ \depth_E M, \depth_E N\} = \depth_E M  \]
 \[\depth_E M \geq \min\{ \depth_E L, \depth_E N\} = \depth_E L. \]
Therefore $\depth_E M = \depth_E L$.
\end{proof}

In practice, it is often easier to specify the singular variety of an $E$-module $M$ as the vanishing set of some collection of polynomials.  We let $S = \kk[x_1, \dots x_n]$ denote the ring of polynomial functions on $E_1$, where $x_i$ is the linear functional dual to $e_i$, and for any ideal $I \subseteq S$, we set $V(I) = \{ \ell \in E_1 \mid f(\ell) = 0 \; \text{for all}\; f \in I\}$.  The following result allows one to compute singular varieties over an exterior algebra with an added variable.

\begin{prop}[{\cite[4.1, 4.2]{LG-quadratic:quotients:of:exterior:algebras}}]
\label{depth:and:base:change}
Let $E'$ be an exterior algebra over an algebraically closed field $\kk$, and let $E = E'\langle e \rangle$ denote an exterior algebra in one more variable.  Denote by $S'$ and $S = S'[x]$ the polynomial rings over $\kk$ dual to $E'$ and $E$ respectively, and let $M$ be a finitely generated graded $E'$-module with $V_{E'}(M) = V(I)$ for some ideal $I \subseteq S'$.  Then:
\begin{enumerate}[label = \textnormal{(\alph*)}]
    \item $V_E(M) = V(IS)$, and $\depth_E M = \depth_{E'} M$.
    \item $V_E(M \tensor_\kk \kk\langle e \rangle) = V(x, IS)$, and $\depth_E M = \depth_{E'} M + 1$.
\end{enumerate}
\end{prop}

For an $E$-module $M$, the \term{complexity} of $M$ is
\[ \cx_E M = \inf \{ c \in \ZZ \mid \beta^E_i(M) \leq \alpha i^{c-1} \text{ for some } \alpha \in \mathbb{R} \text{ and all } i \geq 1 \}. \]
(See Subsection \ref{subsection:free:resolutions} for a review of the Betti numbers $\beta_i^E(M)$ in the commutative case.)  The complexity of $M$ can be thought of as the size of the minimal free resolution of an $E$-module, similar to the projective dimension of an $S$-module.  There is an Auslander-Buchsbaum style theorem relating depth and complexity.

\begin{thm}{\cite[3.2]{resolutions:of:monomial:ideals:over:exterior:algebras}}\label{auslander:buchsbaum}
If $\mathbbm{k}$ is an infinite field and $M$ is a finitely generated $E$-module, then each maximal $M$-regular sequence has $\depth_E M$ elements, and
\[ \depth_E M + \cx_E M = n.\]
\end{thm}

\subsection{Exterior Edge Ideals}

Let $G$ be a finite, simple graph with vertices $v_1, \dots, v_n$.  The \term{exterior edge ideal} of $G$ is the ideal of the exterior algebra $E = \bigwedge_\kk \langle e_1, \dots, e_n \rangle$ defined by:
\[ 
I_E(G) = (e_ie_j \mid \{v_i, v_j\} \;\text{is an edge of}\; G). 
\]

Our aim is to study the depth and singular variety of the quotient ring $E/I_E(G)$ for various families of graphs.  In this subsection, we record some general observations about working with quotients by exterior edge ideals that will be useful in the subsequent sections.

\begin{notation}
Throughout the remainder of this paper, unless explicitly specified otherwise, $G$ will denote a (finite, simple) graph with vertices $v_1, \dots, v_n$, $E = \bigwedge_\kk\langle e_1, \dots, e_n\rangle$ denotes the corresponding exterior algebra over a fixed algebraically closed field $\kk$ of characteristic zero, $I = I_E(G) \subseteq E$ is the exterior edge ideal of $G$, and $S = \kk[x_1, \dots x_n]$ denotes the dual polynomial ring.

For simplicity, we denote the depth, complexity, and singular variety of $E/I_E(G)$ by $\depth_E G$, $\cx_E G$, and $V_E(G)$ respectively. We also write $\pd_S G$ and $\depth_S G$ for the projective dimension and depth of $S/I_S(G)$ over $S$, where $I_S(G) = (x_ix_j \mid \{v_i, v_j\} \;\text{is an edge of}\; G)$ is the more common commutative edge ideal in $S$.  

If $v_i$ and $v_j$ are vertices of $G$, we write $v_iv_j$ to denote that $\{v_i, v_j\}$ is an edge of $G$.  The \term{closed neighborhood} of the vertex $v_i$ is the set $N[v_i] = N(v_i) \cup \{v_i\}$ of neighbors of $v_i$ together with $v_i$.  For any unexplained graph-theoretic terminology, we refer the reader to \cite{West:graph:theory}.
\end{notation}

The following occurs frequently enough as a special case in the arguments of the rest of the paper  that it worth mentioning to avoid future confusion.

\begin{rmk}[The empty graph]
When $G$ is the empty graph, there are no vertices or edges of $G$ so that $E/I_E(G) = E= \kk$.  Since $E_1 = 0$, there are no linear forms with which to form a regular sequence so that $\depth_E G = 0$ and $V_E(G) = 0 = E_1$.
\end{rmk}

\begin{prop} \label{deleting:a:vertex} \label{adding:an:edge}
Let $v_i, v_j$ be a vertices of $G$, and suppose that $v_iv_j$ is an edge of $G$. Then:
\begin{enumerate}[label = \textnormal{(\alph*)}]
    \item $\depth_E G \geq \min\{\depth_E(G \setminus v_i), \depth_E(G \setminus N[v_i])\}$
    \item $\depth_E G \geq \min\{\depth_E(G \setminus v_iv_j), \depth_E(G \setminus (N[v_i] \cup N[v_j]))\}$
\end{enumerate}
\end{prop}

\begin{proof}
If we set $W = N[v_i]$, $G' = G \setminus W$, and $G'' = G \setminus v_i$ and we let $E'$ and $E''$ denote the corresponding exterior algebras on the vertices of $G'$ and $G''$ respectively, then we have a short exact sequence of $E$-modules
\begin{equation} \label{deleting:a:vertex:sequence}
0 \longto \frac{E}{I_{E'}(G')E + (e_j \mid v_j \in W)} \stackrel{e_i}{\longto} \frac{E}{I_E(G)} \longto \frac{E}{I_{E''}(G'')E + (e_i)} \longto 0. 
\end{equation}
The first and last modules in the above exact sequence are isomorphic to $E'/I_{E'}(G')$ and $E''/I_{E''}(G'')$ respectively so that $\depth_E G \geq \min\{\depth_E G', \depth_E G''\}$ by Corollary \ref{depth:sequence:inequality}.  (Note that $\depth_E E'/I_{E'}(G') = \depth_{E'} E'/I_{E'}(G')$ by Proposition \ref{depth:and:base:change} so that there is no harm in writing $\depth_E (G'$.) The second depth inequality follows similarly from the short exact sequence
\[
0 \longto \frac{E}{I_{E'}(G')E + (e_j \mid v_j \in W)} \stackrel{e_ie_j}{\longto} \frac{E}{I_E(G'')} \longto \frac{E}{I_E(G)} \longto 0, 
\]
where $W = N[v_i] \cup N[v_j]$, $G' = G \setminus W$, and $G'' = G \setminus v_iv_j$.
\end{proof}

\begin{prop}\label{disjoint:union}
Suppose $G$ is a disjoint union $G = G' \sqcup G''$ for some subgraphs $G'$, $G''$.  Then
\[
V_E(G) = V_E(G') \cap V_E(G'')
\]
and $\depth_E G = \depth_E G' + \depth_E G''$. 
\end{prop}

\begin{proof}
If $E'$ and $E''$ denote the corresponding exterior algebras on the vertices of $G'$ and $G''$ respectively, then $E/I_E(G) \iso E'/I_{E'}(G') \tensor_\kk E''/I_{E''}(G'')$ so that the statement about singular varieties follows from Theorem \ref{singular:varieties}.  If $S$, $S'$, and $S''$ denote the polynomial rings dual to $E$, $E'$, and $E''$ and we write $V_{E'}(G') = V(I')$ and $V_{E''}(G'') = V(I'')$ for some ideals $I' \subseteq S'$ and $I'' \subseteq S''$, then $V_E(G) = V_E(G') \cap V_E(G'') = V(I'S + I''S)$ by Proposition \ref{depth:and:base:change}.  Since $I'S$ and $I''S$ are ideals generated by polynomials in disjoint sets of variables, it follows that
\[
\depth_E G = \hgt (I'S + I''S) = \hgt I' + \hgt I'' = \depth_E G' + \depth_E G''.  \qedhere
\]
\end{proof}

\begin{prop} \label{identifying:nonadjacent:vertices} \label{contracting:an:edge}
Let $v_i$ and $v_j$ be vertices of $G$, and set
\[
J = (e_k \mid v_k \in N[v_i] \cap N[v_j]) + (e_ae_b \mid v_a \in N(v_i), v_b \in N(v_j)). 
\] 
Then $(I_{E}(G):e_i-e_j) = I_E(G)+(e_i-e_j)+J$.
\end{prop}
\begin{proof}
Set $L = I_E(G) + (e_i - e_j) + J$.  Observe that if $v_k$ is adjacent to both $v_i$ and $v_j,$ then both $e_ke_i$ and $e_ke_j$ are in $I_{E}(G)$ and so $e_k \in (I_E(G):e_i-e_j)$. (This includes the possibility $k = i, j$ since, for example, $e_i^2 = 0$.)  Similarly, if $v_a$ is adjacent to $v_i$ and $v_b$ is adjacent to $v_j,$ then $e_a e_b e_i, e_a e_b e_j \in I_E(G).$ Thus, $e_a e_b \in (I_E(G):e_i-e_j)$.  Since $(e_i - e_j)^2 = 0$, this shows that $L \subseteq (I_E(G):e_i-e_j)$.

Now, let $f\in (I_{E}(G):e_i-e_j)$.  To show that $f \in L$, we may write $f$ as a linear combination of monomials 
\[
f = 
e_ie_j\sum_T \lambda_{T,i,j}e_T + e_i\sum_T \lambda_{T,i}e_T
+e_j\sum_T \lambda_{T,j}e_T + \sum_T \lambda_Te_T,
\] 
where $\lambda_{T, i,j}, \lambda_{T, i}, \lambda_{T, j}, \lambda_T \in \kk$ and $T$ runs over all subsets of $[n]$ not containing $i$ or $j$ in each sum.  Without loss of generality, we may assume that none of the monomials of $f$ belong to $L$.  In particular, since $e_ie_j = (e_i-e_j)e_j \in L$ we can subtract the first sum from $f$ and assume $\lambda_{T, i, j} = 0$ for all $T$. Then we have
\[
(e_i - e_j)f = \sum_T (\lambda_{T, i} + \lambda_{T, j})e_ie_je_T + \sum_T \lambda_Te_ie_T - \sum_T \lambda_Te_je_T,
\]
where all monomials appearing on the right are distinct, and so, they must all belong to $I_E(G)$.  If $\lambda_T \neq 0$, then $e_ie_T, e_je_T \in I_E(G)$ so that either $e_T \in I_E(G)$ or there exist $a, b \in T$ such that $e_ie_a, e_je_b \in I_E(G)$.  Either way, we see that $e_T \in L$ contrary to our assumption, so we must have $\lambda_T = 0$ for all $T$.  

Similarly, if $\lambda_{T,i} + \lambda_{T,j} \neq 0$, then $e_ie_je_T \in I_E(G)$ implies there is an $a \in T$ such that $e_ie_a$ or $e_je_a$ belongs to $I_E(G)$.  Suppose without loss of generality that $e_ie_a \in I_E(G)$.  If $\lambda_{T, i} \neq 0$, then $e_ie_T \in L$ contrary to our assumption on $f$, so we must have $\lambda_{T, i} = 0$ so that $\lambda_{T, j} \neq 0$.  But then $e_je_T$ is a monomial of $f$ divisible by $e_je_a = e_ie_a - (e_i - e_j)e_a \in L$, which also contradicts our assumption.  Hence, we must have $\lambda_{T, i} + \lambda_{T, j} = 0$ for all $T$ so that
\[
f = \sum_T \lambda_{T, i}(e_i - e_j)e_T \in L. \qedhere
\]
\end{proof}

\begin{cor} \label{difference:is:regular}
With the same notation as in the preceding proposition, if $N[v_i] \cap N[v_j] = \emptyset$ and every vertex in $N(v_i)$ is adjacent to every vertex in $N(v_j)$, then $e_i-e_j$ is regular on $E/I_E(G)$. Furthermore, if $G/(v_i \sim v_j)$ denotes the graph obtained by identifying $v_i$ and $v_j$, then
\[
\depth_E G \geq \depth_E G/(v_i \sim v_j) + 1.
\]
\end{cor}

\begin{proof}
The extra assumptions imply that $J \subseteq I_E(G) + (e_i - e_j)$ so that ${(I_E(G) : e_i - e_j)} = I_E(G) + (e_i - e_j)$, which is equivalent to $e_i - e_j$ being regular on $E/I_E(G)$. If we set $H = G/(v_i \sim v_j)$, the second statement follows from the observation that $(I_E(G), e_i - e_j) = (I_E(H), e_i - e_j)$.
\end{proof}

\subsection{Free Resolutions and Graded Betti Numbers} \label{subsection:free:resolutions}

For modules that are quotients of the exterior algebra by a monomial ideal, we can detect depth using tools from commutative algebra whose definitions we recall in this subsection.  

If $I$ is an ideal generated by homogeneous polynomials in a polynomial ring $S = \kk[x_1, \dots, x_n]$ over a field, a \term{free resolution} of the quotient ring $R = S/I$ is an exact sequence of homomorphisms
\[ 
0 \to F_p \stackrel{d_p}{\to} F_{p - 1} \to \cdots \to F_1 \stackrel{d_1}{\to} F_0 \stackrel{d_0}{\to} R \to 0,
\]
where each $F_i \iso S^{r_i}$ is a free $S$-module (that is, the image of each map in the sequence is the kernel of the preceding map).   Concretely, we can represent each map $d_i$ as multiplication by a matrix with entries in $S$.  Every quotient ring $R$ has a free resolution that is \term{minimal} in the sense the entries of the matrices representing the maps $d_i$ are all homogeneous polynomials of positive degree, and this resolution is unique up to isomorphism.  

By imposing a grading on the free modules $F_i$ in the minimal free resolution of $R$ so that the maps $d_i$ preserve degrees, we can write each free module as a direct sum $F_i \iso \bigoplus_{j \in \ZZ} S(-j)^{\beta_{i,j}}$, where $S(-j)$ denotes the free graded $S$-module whose $d$-th graded component is $S(-j)_d = S_{d-j}$.  The numbers $\beta_{i,j}^S(R) := \beta_{i,j}$ that record the number of basis elements of $F_i$ of degree $j$ are then called the \term{graded Betti numbers} of $R$ over $S$.  When $I$ is a monomial ideal, the minimal free resolution of $R$ even admits a finer $\ZZ^n$-grading, where $\deg x_i = \vec{e}_i$ is the $i$-th standard basis vector of $\ZZ^n$.  In that case, for each $\vec{a} = (a_1, \dots, a_n) \in \ZZ^n$, we denote by $\beta^S_{i,\vec{a}}(R)$ the corresponding \term{multigraded Betti number} of $R$ over $S$. In particular, when $I$ is squarefree monomial and $\beta_{i,\vec{a}}^S(R) \neq 0$, it follows from the Taylor resolution, for example, that every entry of $\vec{a}$ is either 0 or 1.  See \cite{graded:syzygies} for further details.

The next result provides the important connection between measuring depth of monomial ideals over an exterior algebra and free resolutions over a polynomial ring.

\begin{thm}[{\cite[3.1, 3.2, 4.2]{resolutions:of:monomial:ideals:over:exterior:algebras}}] \label{depth:via:maximal:shifts}
Let $E$ be an exterior algebra in $n$ variables over an algebraically closed field $\kk$, $J \subseteq E$ be a monomial ideal, and $I$ be the corresponding square-free monomial ideal of the dual polynomial ring $S$.  Denote by $\Sigma$ the set of non-vanishing multidegrees in the minimal free resolution of $S/I$.  Then
\[
V_E(E/J) = \bigcup_{\vec{a} \in \Sigma} V(x_i \mid a_i = 0),
\]
\[
\depth_E E/J = n - \max\{ j \mid \beta^S_{i,j}(S/I) \neq 0 \}.
\]
\end{thm}

\begin{example}\label{ex:freeResolutionOfC3}
Consider the edge ideal $I_E(C_3) = (e_1e_2, e_1e_3, e_2e_3) \subseteq E = \bigwedge_{\kk}\langle e_1, e_2, e_3 \rangle$, and let $I_S(C_3) = (x_1x_2, x_1x_3, x_2x_3) \subseteq S = \kk[x_1, x_2, x_3]$ denote the corresponding edge ideal in the dual polynomial ring.  It can be checked that the minimal free resolution of $R = S/I_S(C_3)$ is
\begin{center}
\begin{tikzcd}[ampersand replacement = \&, column sep = 2.5 em]
0 \rar \& S(-1,-1,-1)^2 \rar{d_2} \& {\Large \substack{S(-1,-1,0) \\ \oplus \\ S(-1,0,-1) \\ \oplus \\ S(0,-1,-1)}} \rar{d_1} \& S\rar \& R \rar \& 0, 
\end{tikzcd}
\end{center}
where
\[
d_2 = \begin{pmatrix} x_3 & 0 \\ -x_2 & x_2 \\ 0 & -x_1 \end{pmatrix} \qquad
d_1 = \begin{pmatrix} x_1x_2 & x_1x_3 & x_2x_3 \end{pmatrix}
\]
so that the only nonzero multidegrees are $\Sigma = \{(1,1,0), (1,0,1), (0,1,1), (1,1,1)\}$ and the only nonzero graded Betti numbers of $R$ are $\beta_{0,0}^S(R) = 1$, $\beta_{1,2}^S(R) = 3$ and $\beta_{2,3}^S(R) = 2$.  Hence, the above theorem implies that 
\[ V_E(C_3) = V(0) = E_1 \]  
\[\depth_E C_3 = 3 - 3 = 0.\]  
We will see that this example is greatly generalized in Theorem \ref{depth:cycles} below.
\end{example}

\begin{cor}\label{depth:compare} Let $G$ be a graph on $n$ vertices, and let $E$ and $S$ be the corresponding exterior algebra and polynomial ring over $\kk$ as above.  Then
\[\depth_E G \leq \depth_S G - 1\]
with equality if $I_S(G)$ has a linear free resolution over $S$.
\end{cor}

\begin{proof} As $I_S(G)$ is generated by quadratic monomials, $\beta_{i,j}^S(S/I_S(G)) = 0$ for $i > 0$ and $j \le i$.  Thus $\beta_{p,j}^S(S/I_S(G)) \neq 0$ for some $j \ge p+1$, where $p = \pd_S S/I_S(G)$.  Thus 
\[\depth_E G = n - \max\{ j \mid \beta^S_{i,j}(S/I_S(G)) \neq 0 \} \leq n - p - 1 = \depth_S G - 1,\]
where the last equality is the Auslander-Buchsbaum Formula \cite[15.3]{graded:syzygies}.  When $I_S(G)$ has a linear free resolution, equality follows from Theorem~\ref{depth:via:maximal:shifts}.
\end{proof}

\begin{rmk}
We note that the above inequality is both tight and the difference $\depth_S G - \depth_E G$ can be arbitrarily large.  For example, since $I_S(K_n)$ is Cohen-Macaulay of height $n - 1$, it will follow from Corollary \ref{complete:multipartite:graphs} that 
\[
\depth_S K_n - 1 = n -1 - \hgt I_S(K_n) = 0 = \depth_E K_n.
\]
On the other hand, if $C_n$ denotes the cycle graph on $n$ vertices, then $\depth_E C_n = 0$ by Theorem \ref{depth:cycles}, yet 
\[\depth_S C_n = \begin{cases} \floor{\frac{n}{3}} & \text{if $n \equiv 1 \pmod{3}$} \\[1ex]
\floor{\frac{n+2}{3}} & \text{if $n \not\equiv 1 \pmod{3}$} \end{cases}\]
by \cite[5.0.6]{Bouchat}. Thus, the difference $\depth_S G - \depth_E G$ can be arbitrarily large.
\end{rmk}

\section{Some Singular Variety Computations}\label{section:singularvariety}

In this section, we compute the singular varieties and depths of several large classes of graphs, including paths, cycles, spider graphs, star graphs, complete multipartite graphs, and trees.  We also investigate the effect whiskering a graph has on its depth.

\subsection{Trees and Cycles}

\begin{thm} \label{adding:a:leaf}
Let $G$ be a graph with a vertex $w = v_n$ of degree one, and let $v$ be the unique vertex adjacent to $w$.  Then
\[
V_E(G) = (V_E(G \setminus w) \cap V(x_n)) \cup V_E(G \setminus N[v]) 
\]
so that
\[
\depth_E G = \min\{\depth_E(G \setminus w) + 1, \depth_E(G \setminus N[v]) \}. 
\]
\end{thm}

\begin{proof}
Let $\Sigma(G)$ denote the set of multidegrees of basis elements for any free module in the minimal free resolution of $S/I_S(G)$ over $S$.  We may assume the vertices of $G$ have been numbered so that $v_1, \dots, v_r$ are the neighbors of $v$ different from $w$, $v = v_{r+1}$, and $w = v_n$.  We also set $G' = G \setminus w$, $S' = \kk[x_1, \dots, x_{n-1}]$, and $S'' = \kk[x_{r+2},\dots, x_{n-1}]$.  We then have a short exact sequence of graded $S$-modules
\[
0 \longto S/(I_S(G') : x_{r+1}x_n)(-2) \stackrel{x_{r+1}x_n}{\longto} S/I_{S'}(G')S \longto S/I_S(G) \longto 0, 
\]
where $(I_S(G') : x_{r+1}x_n) = I_{S''}(G \setminus N[v])S + (x_1,\dots, x_r)$.  Let $F_\bullet$ and $F'_\bullet$ denote the minimal free resolutions of $S/(I_S(G'): x_{r+1}x_n)$ and $S/I_{S'}(G')S$ respectively.  We note that $F'_\bullet$ is obtained from the minimal free resolution of $S'/I_{S'}(G')$ over $S'$ by tensoring with $S$ since $S$ is free over $S'$ so that the set of multidegrees of basis elements for any free module in $F'_\bullet$ is precisely $\Sigma(G')$ (after identifying $\ZZ^{n-1}$ with $\ZZ^{n-1} \times \{0\} \subseteq \ZZ^n$).  Additionally, since 
\[
S/(I_S(G') : x_{r+1}x_n) \iso S''/I_{S''}(G \setminus N[v]) \tensor_\kk \kk[x_1,\dots, x_{r+1}, x_n]/(x_1,\dots,x_r),
\]
we know that $F_\bullet$ can be obtained as the tensor product over $\kk$ of the minimal free resolution of $S''/I_{S''}(G \setminus N[w])$ over $S''$ with the Koszul complex on the variables $x_1, \dots, x_r$.  Hence, if $\Delta = \{0,1\}^r$, we see that the set of multidegrees of basis elements for any free module in $F_\bullet$ is 
\[
\Sigma(G \setminus N[v]) + \Delta = \{\vec{a} + \vec{b} \mid \vec{a} \in \Sigma(G \setminus N[v]), \vec{b} \in \Delta\},
\]  
where we again identify $\ZZ^r$ with $\ZZ^r \times \{0\}^{n-r} \subseteq \ZZ^n$.

If $\phi: F_\bullet(-2) \to F'_\bullet$ is any chain map lifting $S/(I_S(G') : x_{r+1}x_n)(-2) \stackrel{x_{r+1}x_n}{\longto} S/I_{S'}(G')S$, then as observed in \cite[2.1.1]{Bouchat}, the mapping cone of $\phi$ provides a minimal free resolution of $S/I_S(G)$ since every multidegree $\vec{a} = (a_1, \dots, a_n)$ appearing as a shift in $F'_\bullet$ must have $a_n = 0$.  Since the $i$-th free module in the mapping cone of $\phi$ is just $F'_i \oplus F_{i-1}$ for $i > 0$ and $S$ if $i = 0$, it follows that
\[ \Sigma(G) = \Sigma(G') \cup (\Sigma(G \setminus N[v]) + \Delta + \vec{e}_{r+1} + \vec{e}_n). \]
Hence, Theorem \ref{depth:via:maximal:shifts} and Proposition \ref{depth:and:base:change} imply that
\begin{align*}
    V_E(G) &= \left(\bigcup_{\vec{a} \in \Sigma(G')} V(x_i, x_n \mid a_i = 0)\right) \cup \left(\bigcup_{\vec{b} \in \Sigma(G \setminus N[v])} V(x_i \mid b_i = 0) \right) \\[2 ex]
    &= (V_E(G \setminus w) \cap V(x_n)) \cup V_E(G \setminus N[v]),
\end{align*}
so that
\begin{align*}
    \depth_E G &= \codim_{E_1} V_E(G) \\
    &= \min\{ \codim_{E_1} (V_E(G \setminus w) \cap V(x_n)), \codim_{E_1} V_E(G \setminus N[v])\} \\
    &= \min\{\depth_E(G \setminus w) + 1, \depth_E(G \setminus N[v])\}. \qedhere
\end{align*}
\end{proof}

As a consequence of the preceding theorem, we recover the following result of McCullough and Mere on the singular varieties of path graphs.

\begin{cor}[{\cite[4.3]{LG-quadratic:quotients:of:exterior:algebras}}] \label{depth:path:graphs}
Let $P_n$ denote the path graph on $n \geq 1$ vertices with edges $v_iv_{i+1}$ for all $i < n$.
\begin{enumerate}[label = \textnormal{(\alph*)}]
    \item If $n \equiv 1 \pmod{3}$, then $\depth_E P_n = 1$ and 
    \[ V_E(P_n) = \bigcup_{i\, \equiv\, 1\!\!\!\!\!\pmod{3}} \!\!\!V(x_i). \]
    \item If $n \not\equiv 1 \pmod{3}$, then $\depth_E P_n = 0$ and $V_E(P_n) = E_1$.  
\end{enumerate}
\end{cor}

\begin{proof}
For the $n \leq 3$, the result is a straightforward computation, which we omit.  If $n \geq 4$ and the result holds for $k < n$, then applying the preceding theorem with $w = v_n$ yields $V_E(P_n) = (V_E(P_{n-1}) \cap V(x_n)) \cup V_E(P_{n-3})$.  If $n \not\equiv 1 \pmod{3}$, then $V_E(P_{n-3}) = E_1$ by Proposition \ref{depth:and:base:change} and induction so that $V_E(P_n) = E_1$ and $\depth_E P_n = 0$.  If $n \equiv 1 \pmod{3}$, then 
\[
V_E(P_{n-3}) = \bigcup_{\substack{i \equiv 1 \!\!\!\pmod{3} \\ i < n}} V(x_i) \qquad \text{ and }  \qquad V(P_{n-1}) = E_1
\] 
so that \[ V_E(P_n) = V(x_n) \cup \left(\bigcup_{\substack{i \equiv 1 \!\!\!\pmod{3} \\ i < n}} V(x_i)\right) = \bigcup_{i \equiv 1 \!\!\!\pmod{3}} V(x_i), \]
whence $\depth_E P_n = 1$.
\end{proof}

Since every tree has a degree one vertex, we also note that the above theorem can be used to recursively compute the depth or singular variety of any tree. 

\begin{example}
Consider the tree $G$ shown below.
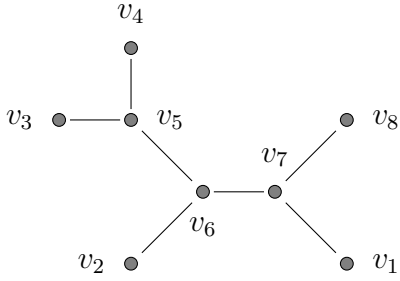
\begin{center}
    \graph[scale = 0.85]{ 
    4.5/-1.5/v_1/0,
    0/-1.5/v_2/180,
    -1.5/1.5/v_3/180,
    0/3/v_4/90,
    0/1.5/v_5/0,
    1.5/0/v_6/-90,
    3/0/v_7/90,
    4.5/1.5/v_8/0}
    {
    v_3/v_5,
    v_4/v_5,
    v_5/v_6,
    v_2/v_6,
    v_6/v_7,
    v_7/v_1,
    v_7/v_8}
\end{center}
The table in Figure \ref{depth:of:tree} shows two sequences of subgraphs obtained by deleting one whisker and its closed neighborhood at a time until we obtain a pair of paths for which the depth is known by Corollary \ref{depth:path:graphs}.  We then apply Theorem \ref{adding:a:leaf} to recursively compute the depth of the original tree.
\begin{figure}[ht]
\begin{center}
\vspace{2 ex}
\begin{tabular}{|c|c|c|c|c|}
\hline
$i$ & $G_i = G_{i-1} \setminus v_i$ & $\depth_E G_i$ & $G'_i = G_{i-1} \setminus N[v_i]$ & $\depth_E G'_i$ \\
\hline
0 & \begin{minipage}{0.3\textwidth}
\vspace{1 ex}
\graph[scale = 0.85]{ 
    3/-1/v_1/0,
    0/-1/v_2/180,
    -1/1/v_3/180,
    0/2/v_4/90,
    0/1/v_5/0,
    1/0/v_6/-90,
    2/0/v_7/90,
    3/1/v_8/0}
    {
    v_3/v_5,
    v_4/v_5,
    v_5/v_6,
    v_2/v_6,
    v_6/v_7,
    v_7/v_1,
    v_7/v_8}
\vspace{1 ex}
\end{minipage}
& 1 & -- & -- \\
\hline
1 & \begin{minipage}{0.3\textwidth}
\begin{center}
\vspace{1 ex}
\graph[scale = 0.85]{ 
    0/-1/v_2/180,
    -1/1/v_3/180,
    0/2/v_4/90,
    0/1/v_5/0,
    1/0/v_6/-90,
    2/0/v_7/90,
    3/1/v_8/0}
    {
    v_3/v_5,
    v_4/v_5,
    v_5/v_6,
    v_2/v_6,
    v_6/v_7,
    v_7/v_8}
    \vspace{1 ex}
\end{center}
\end{minipage} & 1 
& \begin{minipage}{0.3\textwidth}
\begin{center}
\graph[scale = 0.85]{ 
    -1/-1/v_3/180,
    -1/1/v_4/180,
    -1/0/v_5/0,
    1/0/v_2/0}
    {
    v_3/v_5,
    v_4/v_5}
\end{center}
\end{minipage} & 1 \\
\hline
2 & \begin{minipage}{0.3\textwidth}
\begin{center}
\vspace{1 ex}
\graph[scale = 0.85]{ 
    -1/1/v_3/180,
    0/2/v_4/90,
    0/1/v_5/0,
    1/0/v_6/-90,
    2/0/v_7/90,
    3/1/v_8/0}
    {
    v_3/v_5,
    v_4/v_5,
    v_5/v_6,
    v_6/v_7,
    v_7/v_8}
\vspace{1 ex}
\end{center}
\end{minipage} & 0 
& \begin{minipage}{0.3\textwidth}
\begin{center}
\graph[scale = 0.85]{ 
    -2/0/v_3/90,
    0/0/v_4/90,
    2/0/v_8/90}
    {}
\end{center}
\end{minipage} & 3 \\
\hline
3 & \begin{minipage}{0.3\textwidth}
\begin{center}
\vspace{1 ex}
\graph[scale = 0.85]{ 
    -1/0/v_4/90,
    0/0/v_5/-90,
    1/0/v_6/90,
    2/0/v_7/-90,
    3/0/v_8/90}
    {
    v_4/v_5,
    v_5/v_6,
    v_6/v_7,
    v_7/v_8}
\vspace{1 ex}
\end{center}
\end{minipage} & 0 
& \begin{minipage}{0.3\textwidth}
\begin{center}
\vspace{1 ex}
\graph[scale = 0.85]{ 
    -1/0/v_7/180,
    1/0/v_8/0}
    {v_7/v_8}
\vspace{1 ex}
\end{center}
\end{minipage} & 0 \\
\hline
\end{tabular}
\caption{Computing Depth by Trimming Whiskers}
\label{depth:of:tree}
\end{center}
\end{figure}
In this case, we have strategically chosen the vertices to delete so that the graphs $G'_i$ have immediately recognizable depth, but in general, one may obtain a repeatedly branching tree of subgraphs.
\end{example}

The following modification of the depth equality in the theorem will be useful later.

\begin{cor} \label{deleting:vertex:with:leaf}
Let $G$ be a graph with a vertex $w$ of degree one, and let $v$ be the unique vertex adjacent to $w$.  Then
\[
\depth_E G = \min\{\depth_E(G \setminus v), \depth_E(G \setminus N[v]) \}. 
\]
\end{cor}

\begin{proof}
By Proposition \ref{deleting:a:vertex}, we know that
\[
\depth_E G \geq \min\{\depth_E(G \setminus v), \depth_E(G \setminus N[v])\}.
\]
In the case where $\depth_E(G \setminus v)$ and $\depth_E(G \setminus N[v])$ are different, we get the desired result via Corollary \ref{depth:sequence:inequality}.

So, suppose $\depth_E(G \setminus v) = \depth_E(G \setminus N[v])$.  Applying Proposition \ref{deleting:a:vertex} again to the graph $G \setminus w$, we have
\[
\depth_E(G \setminus w) \geq \min\{\depth_E(G \setminus \{v, w\}), \depth_E(G \setminus N[v])\}.
\]
Since $G \setminus v = (G \setminus \{v,w\}) \sqcup \{w\}$, we have $\depth_E(G \setminus v) = \depth_E(G \setminus \{v,w\}) + 1$. Since the depths of $G \setminus v$ and $G \setminus N[v]$ were assumed to be equal and $G \setminus \{v,w\}$ has depth one less than that of $G \setminus v$, we get that $\depth_E(G \setminus w)=\depth_E(G \setminus v)-1$ by Corollary \ref{depth:sequence:inequality}. Finally, we use Theorem \ref{adding:a:leaf} to calculate the depth of $G$:
\begin{align*}
\depth_E G &= \min\left(\depth_E(G \setminus w) + 1, \depth_E(G \setminus N[v]) \right) \\
             &= \min\left(\depth_E(G \setminus v), \depth_E(G \setminus N[v]) \right). \qedhere
\end{align*}
\end{proof}

A \term{spider graph} is a tree that has exactly one vertex of degree greater than 2. This vertex will be called the \term{head}. Since no other vertex can have degree bigger than 2, the spider will have induced paths coming off from the head. We will call these paths the \term{legs}. 

Denote by $SP_{n_1, \dots , n_k}$ a spider with legs of lengths $1 \leq n_1 \leq n_2 \leq \dots \leq n_k$, and label the vertices of the leg of length $n_r$ as $v_0 = v_{r,0}, v_{r,1}, \dots, v_{r, n_r}$ where $v_0$ is the head of the spider.

\begin{thm}
Let $SP_{n_1, \dots , n_k}$ be a spider graph, and let $E = \bigwedge_{\kk}{\langle e_0, e_{r,i} \mid 1 \leq r \leq k, 1 \leq i \leq n_r \rangle}$ be the corresponding exterior algebra. Then
 \[ V_E(SP_{n_1, \dots , n_k}) = \left (\bigcap_{n_r \equiv 1\!\!\!\!\!\pmod{3}} \bigcup_{i \equiv 1\!\!\!\!\!\pmod{3}} \!\!\!\!\!V(x_{r,i})\right ) \cup \left ( \bigcap_{n_r \equiv 2\!\!\!\!\!\pmod{3}} \bigcup_{i \equiv 2\!\!\!\!\!\pmod{3}} \!\!\!\!\!V(x_{r,i}) \right ), \]
and $\depth_E SP_{n_1, \dots , n_k} = \min\{p, q\}$, where $p$ is the number of legs whose length is congruent to 1 mod 3 and $q$ is the number of legs whose length is congruent 2 mod 3.
\end{thm}
 
 \begin{proof}
Set $G = SP_{n_1, \dots , n_k}$.  If $n_r = 1$ for all $r$, then $G = K_{1,k}$ is a complete bipartite graph, and the statement of the theorem asserts that $\depth_E G = 0$ and $V_E(G) = E_1$ since the rightmost intersection is an intersection over an empty collection.  This case follows from Corollary \ref{complete:multipartite:graphs}, so we may assume that $n_r \geq 2$ for some $r$.  

By deleting the head $v_0$ of the spider, we obtain an exact sequence of $E$-modules as \eqref{deleting:a:vertex:sequence} where $G \setminus v_0 = P_{n_1} \sqcup \dots \sqcup P_{n_k}$ and $G \setminus N[v_0] = P_{n_1 - 1} \sqcup \dots \sqcup P_{n_k - 1}$.  By combining Proposition \ref{disjoint:union}, Proposition \ref{depth:and:base:change}, and Corollary \ref{depth:path:graphs}, we compute the respective singular varieties
 \[ V_E(G \setminus v_0) = 
 \bigcap_{r=1}^k V_E(P_{n_r}) =
 \!\!\!\!\!\!\!\bigcap_{n_r \equiv 1\!\!\!\!\!\pmod{3}} \!\!\!\!\!\!V_E(P_{n_r}) =
 \!\!\!\bigcap_{n_r \equiv 1\!\!\!\!\!\pmod{3}} \left ( \bigcup_{i \equiv 1\!\!\!\!\!\pmod{3}} \!\!\!\!\!V(x_{r,i}) \right ), \]
 and
 \[ V_E(G \setminus N[v_0]) = 
 \bigcap_{r=1}^k V_E(P_{n_r -1}) =
 \!\!\!\!\!\!\!\bigcap_{n_r \equiv 2\!\!\!\!\!\pmod{3}} \!\!\!\!\!\!V_E(P_{n_r -1}) =
 \!\!\!\bigcap_{n_r \equiv 2\!\!\!\!\!\pmod{3}} \left ( \bigcup_{i \equiv 2\!\!\!\!\!\pmod{3}} \!\!\!\!\!V(x_{r,i}) \right) \]
 where we choose the variables $x_{r,i}$ with $i \equiv 2 \pmod{3}$ in the last equality since the $r$-th component of $G \setminus N[v_0]$ has vertices labeled by $v_{r,2}, v_{r,3}, \dots, v_{r,n_r}$.  
 
 In the former case, we can rewrite $V_E(G \setminus v_0)$ as a union of linear subspaces determined by the vanishing of $p$ variables $x_{r,i_r}$ where for each $r$ with $n_r \equiv 1 \pmod{3}$ we choose one $i_r \leq n_r$ with $i_r \equiv 1 \pmod{3}$.  Similarly, we can rewrite $V_E(G \setminus N[v_0])$ as a union of linear subspaces determined by the vanishing of $q$ variables $x_{r,i_r}$ where for each $r$ with $n_r \equiv 2 \pmod{3}$ we choose one $i_r \leq n_r$ with $i_r \equiv 2 \pmod{3}$. By Theorem \ref{singular:varieties}, we know that $V_E(G) \subseteq V_E(G \setminus v_0) \cup V_E(G \setminus N[v_0])$.  However, we also know that $V_E(G \setminus v_0) \subseteq V_E(G) \cup V_E(G \setminus N[v_0])$, and since it is clear that none of the linear subspaces whose union is $V_E(G \setminus v_0)$ is contained in any of the linear subspaces whose union is $V_E(G \setminus N[v_0])$, we must have $V_E(G \setminus v_0) \subseteq V_E(G)$.  An analogous argument shows that $V_E(G \setminus N[v_0]) \subseteq V_E(G)$ as well so that
 \[ V_E(G) = V_E(G \setminus v_0) \cup V_E(G \setminus N[v_0]).\]
 Finally, because every component of $V_E(G \setminus v_0)$ has codimension $p$ and every component of $V_E(G \setminus N[v_0])$ has codimension $q$, it follows that
 \[\depth_E G = \codim_{E} V_E(G) = \min\{p, q\}. \qedhere \]
 \end{proof}
 
 \begin{example}
The graph $SP_{1, 2,3,4}$ is shown below.
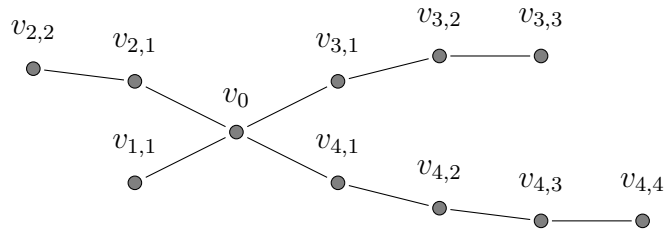
\begin{center}
    \graph[scale = 0.9]{
    0/0/v_0/90,
    -2/-1/v_{1,1}/90,
    -2/1/v_{2,1}/90,
    -4/1.25/v_{2,2}/90,
    2/1/v_{3,1}/90,
    4/1.5/v_{3,2}/90,
    6/1.5/v_{3,3}/90,
    2/-1/v_{4,1}/90,
    4/-1.5/v_{4,2}/90,
    6/-1.75/v_{4,3}/90,
    8/-1.75/v_{4,4}/90}{
    v_0/v_{1,1},
    v_0/v_{2,1},
    v_{2,1}/v_{2,2},
    v_0/v_{3,1},
    v_{3,1}/v_{3,2},
    v_{3,2}/v_{3,3},
    v_0/v_{4,1},
    v_{4,1}/v_{4,2},
    v_{4,2}/v_{4,3},
    v_{4,3}/v_{4,4}}
\end{center}
Since $SP_{1,2,3,4}$ has 2 legs of length congruent to 1 mod 3 and 1 leg of length congruent to 2 mod 3, the previous proposition shows that $\depth_E SP_{1, 2,3,4} = \min(2,1) = 1$ and
\begin{align*}
    V_E(SP_{1, 2,3,4}) &= V(x_{1,1})\cap (V(x_{4,1})\cup V(x_{4,4})) \cup                                      \left(V(x_{2,2})\right) \\
                       &= V(x_{1,1},x_{4,1}) \cup V(x_{1,1},x_{4,4}) \cup V(x_{2,2}).
\end{align*}
\end{example}

Having seen how to compute the depth and singular varieties of trees, we now turn our attention to cycles.

\begin{thm} \label{depth:cycles}
Let $C_n$ denote a cycle graph on $n \geq 3$ vertices.  Then $\depth_E C_n = 0$.
\end{thm}

\begin{proof}
We assume that the vertices $v_1, \dots, v_n$ of $C_n$ have been labeled so that the edges are $v_iv_{i+1}$ for $i < n$ and $v_1v_n$, and we consider cases based on the residue of $n$ mod 3.  If $n \equiv 1  \pmod{3}$ or $n \equiv 2  \pmod{3}$, applying Proposition \ref{deleting:a:vertex} to the deletion of the vertex $v_n$ yields 
\[
\depth_E C_n \geq \min\{\depth_E P_{n-1}, \depth P_{n-3}\}.
\]
When $n \equiv 1 \pmod{3}$, Proposition \ref{depth:path:graphs} yields that $\depth_E P_{n-3} = 1$ and $\depth_E P_{n-1} = 0$. Consequently, $\depth_E C_n = 0$ by Corollary \ref{depth:sequence:inequality}.  When $n \equiv 2 \pmod{3}$, $\depth_E P_{n-3} = 0$ and $\depth_E P_{n-1} = 1$, so the result follows similarly. 

Now, suppose $n \equiv 0 \pmod{3}$. We may further assume that $n \geq 6,$ since the case $n=3$ is handled by Proposition \ref{complete:multipartite:graphs}. Let $G$ denote the graph obtained from $C_n$ by adding the edge $v_1v_3$.  Applying Proposition \ref{adding:an:edge} to the deletion of the edge $v_1v_3$ yields
\[
\depth_E C_n \geq \min\{\depth_E P_{n-5}, \depth_E G\}.
\]
Since $\depth_E P_{n-5} = 1$ by Proposition \ref{depth:path:graphs}, we aim to show $\depth_E G = 0$ in order to prove that $\depth C_n = 0$ by Corollary \ref{depth:sequence:inequality}. Applying Proposition \ref{adding:an:edge} again to the deletion of the edge $v_1v_n$
yields
\[
\depth_E G \geq \min\{\depth_E H, \depth_E P_{n-5}\},
\]
where $H = G \setminus v_1v_n$ is the following graph.
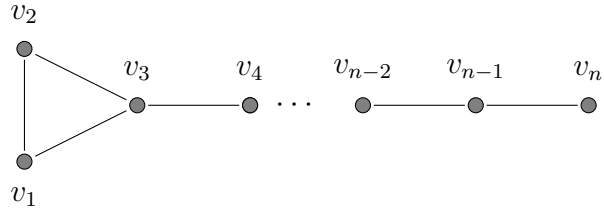
\begin{center}
    \labeledgraph{ 
    -2/-1/v_1/v_1/270,
    -2/1/v_2/v_2/90,
    0/0/v_3/v_3/90,
    2/0/v_4/v_4/90,
    2/0/v_5/\cdots/0,
    4/0/v_{n-2}/v_{n-2}/90,
    6/0/v_{n-1}/v_{n-1}/90,
    8/0/v_n/v_n/90}
    {v_1/v_2,
    v_2/v_3,
    v_3/v_1,
    v_3/v_4,
    v_{n-2}/v_{n-1},
    v_{n-1}/v_n}
\end{center}
Hence, it further suffices to show that $\depth_E H = 0$. Deleting the vertex $v_4$ from $H$ yields 
\[ 
\depth_E H \geq \min\{ \depth_E(P_2 \sqcup P_{n-5}), \depth_E(C_3 \sqcup P_{n-4}) \}.
\]
By Proposition \ref{disjoint:union}, Proposition \ref{depth:path:graphs}, and Corollary \ref{complete:multipartite:graphs}, we see that $\depth_E(P_2 \sqcup P_{n-5}) = 0 + 1 = 1$ and $\depth_E(C_3 \sqcup P_{n-4}) = 0 + 0 = 0$. Hence, $\depth_E H = 0$ by Corollary \ref{depth:sequence:inequality} as wanted.
\end{proof}

\subsection{Duplicating and Coning}

A vertex of a graph which is connected to every other vertex is called \term{universal}, and graphs with a universal vertex are called \term{cones}. Proposition \ref{depth:cone:graphs} states that all cones have depth 0. Special cases of cones include complete graphs, wheel graphs, and windmill graphs. 

\begin{prop}\label{depth:cone:graphs}
If $G$ is a graph with at least two vertices having a universal vertex $v$, then $\depth_E G = 0$. 
\end{prop}

\begin{proof}
Let $u$ be any vertex of $G$ different from $v$, and let $H$ denote the graph obtained by adding a whisker $vw$ to the graph $G \setminus uv$.  By Theorem \ref{adding:a:leaf}, we see that 
\[ 
\depth_E H = \min\{ \depth_E(G \setminus uv) + 1, \depth_E \{u\} \} = 1.
\]
We note that $G$ is isomorphic to the graph obtained from $H$ by identifying the vertices $u$ and $w$.  Since $N[w] \cap N[u] = \emptyset$ and every vertex in $N(w) = \{v\}$ is adjacent to every vertex in $N(v)$ by assumption, Corollary \ref{difference:is:regular} implies that $\depth_E H \geq	\depth_E G + 1$. Since $\depth_E H = 1$, it follows that $\depth_E G = 0$.
\end{proof}

A vertex $v'$ is called a \term{duplicate} of another vertex $v$ if $N(v')=N(v)$.

\begin{prop} \label{depth:duplicate:vertex}
If $G$ is a graph with a duplicate vertex $v$ and $\depth_E(G \setminus v) = 0$, then $\depth_E G = 0$.
\end{prop}

\begin{proof}
If $H = G \setminus v$ and $w$ denotes a duplicate of $v$ in $G$, we note that $G \setminus N[v] = (H \setminus N[w]) \sqcup \{w\}$.  Applying Proposition \ref{deleting:a:vertex} to the deletion of $v$ as well as Proposition \ref{disjoint:union} yields
\[
\depth_E G \geq \min\{\depth_E(H \setminus N[w]) + 1, \depth_E H\}.
\]
Since $\depth_E H = 0$ by hypothesis, Corollary \ref{depth:sequence:inequality} implies that $\depth_E G = \depth_E H = 0$ as desired.
\end{proof}

A \term{complete graph} with $n$ vertices is a graph $K_n$ in which every vertex is universal. A \textbf{complete $r$-partite graph} $K_{n_1,\ldots,n_r}$ is a graph whose vertices can be partitioned into $r$ disjoint sets of cardinalities $n_1,\ldots,n_r$ so that no two vertices in the same set are adjacent, and there is an edge between every pair of vertices from different sets. Complete multipartite graphs include complete graphs since $K_{1,\ldots,1} = K_r$ and \term{star graphs}, which are trees of the form $K_{1, n}$.

 By duplicating a universal vertex, we can combine the previous two results with to obtain the following.

\begin{cor} \label{complete:multipartite:graphs}
For $K_{n_1, \dots, n_r} \neq K_1$, we have ${\depth_E K_{n_1, \dots, n_r} = 0}$.
\end{cor}

\begin{proof}
We prove the claim by induction on the size of $n_1$, and we note that the base case is handled by Proposition \ref{depth:cone:graphs} since $K_{n_1, \dots, n_r} \neq K_1$. So assume $n_1 \geq 2$ and that the corollary holds for all complete multipartite graphs $K_{m_1, \dots, m_s}$ with $m_1 < n_1$.  If $v$ is a vertex in the set of size $n_1$ in the partition of the vertices of $K_{n_1, \dots, n_r}$, then $v$ is duplicate vertex since $n_1 \geq 2$ and $K_{n_1, \dots, n_r} \setminus v = K_{n_1-1,n_r \dots, n_r}$ has depth 0 by induction, so the result follows from Proposition \ref{depth:duplicate:vertex}.
\end{proof}

\subsection{Whiskering}

Given a graph $G$,  define $W(G,m)$ to be the graph formed by attaching $m$ new degree 1 vertices to each vertex of $G$. Then $W(G,m)$ is called a \term{whiskering} of $G$ and the degree 1 vertices are called \term{whiskers}.

\begin{thm} \label{multiple:whiskering:depth}
Let $G$ be a graph on $n$ vertices, and let $m > 0$.  Then 
\[ 
\depth_E W(G, m) = m\depth_E W(G, 1). 
\]
\end{thm}

\begin{proof}
We will argue by induction on $n \geq 0$. When $n = 0$, both $W(G, m)$ and $W(G, 1)$ are the empty graph, which has depth zero.

Now, suppose that $n \geq 1$ and that the result holds for all graphs with fewer than $n$ vertices. Let $v$ a vertex of $G$ of degree $d$. Observe that by Corollary \ref{deleting:vertex:with:leaf} we have
\begin{align*}
\depth_E W(G,1) &= \min\{\depth_E(W(G,1) \setminus v), \depth_E(W(G,1) \setminus N[v]) \} \\[1ex]
             &= \min\{1 + \depth_E(W(G \setminus v,1)), d + \depth_E(W(G\setminus N[v],1)) \}.
\end{align*}
Using Corollary \ref{deleting:vertex:with:leaf} again and our inductive assumption, we then calculate
\begin{align*}
\depth_E W(G,m) &= \min\{\depth_E(W(G,m) \setminus v), \depth_E(W(G,m) \setminus N[v]) \} \\[1ex]
             &= \min\{m + \depth_E W(G \setminus v,m), dm + \depth_E W(G\setminus N[v],m) \} \\[1ex]
             &= \min\{m + m\depth_E W(G \setminus v,1), dm + m\depth_E W(G\setminus N[v],1) \} \\[1ex]
             &= m\min\{1 + \depth_E W(G \setminus v,1), d + \depth_E W(G\setminus N[v],1) \} \\[1ex]
             &= m\depth_E W(G,1). \qedhere
\end{align*}
\end{proof}

Given a graph $G$, a set of vertices of $G$ is an \term{independent set} if no two vertices in the set are adjacent. The \term{independence number} $\alpha(G)$ is the largest number of vertices in any independent subset of $G$.

\begin{thm}\label{depth:one:whisker}
Let $G$ be a graph with $n$ vertices. Then $\depth_E W(G,1) = n-\alpha(G)$.
\end{thm}

\begin{proof}
We will argue by induction on $n \geq 0$. When $n = 0$, $W(G, 1)$ is the empty graph so that $\depth_E W(G, 1) = 0 = 0 - 0 = n - \alpha(G)$.  

Now, suppose that $n \geq 1$ and that the result holds for all graphs with fewer than $n$ vertices. Let $A$ be a maximal independent set of $G$, and let $v$ be a vertex in $A$ with degree $d$ in $G$. By Corollary \ref{deleting:vertex:with:leaf}, we have
\begin{align*}
\depth_E W(G,1) &= \min\{\depth_E(W(G,1) \setminus v), \depth_E(W(G,1) \setminus N[v]) \} \\[1ex]
             &= \min\{1 + \depth_E W(G \setminus v,1), d + \depth_E W(G \setminus N[v],1) \} \\[1ex]
             &= \min\{1 + n-1-\alpha(G \setminus v), d + n-d-1-\alpha(G \setminus N[v]) \} \\[1ex]
             &= \min\{n-\alpha(G \setminus v), n-1-\alpha(G \setminus N[v]) \} \\[1ex]
             &= n - \max\{\alpha(G \setminus v), 1+\alpha(G \setminus N[v]) \}.
\end{align*}
If $A'$ is an independent subset of $G \setminus v$, then it is also independent in $G$ because $G \setminus v$ is an induced subgraph of $G$, so $\alpha(G \setminus v) \leq \alpha(G)$. If $A''$ is independent in $G \setminus N[v]$, then $A'' \cup \{v\}$ is independent in $G$, so $1 + \alpha(G \setminus N[v]) \leq \alpha(G)$. Further, $A \setminus \{v\}$ is independent in $G \setminus N[v]$ so $1 + \alpha(G \setminus N[v]) = \alpha(G)$. Hence
\[
\depth_E W(G,1) = n - \max\{\alpha(G \setminus v), 1+\alpha(G \setminus N[v]) \} = n - \alpha(G). \qedhere
\]
\end{proof}

A \term{vertex cover} of a graph $G$ is a set $C$ of vertices of $G$ such that every edge has one of its ends in $G$.  A vertex cover $C$ is \term{minimal} if no proper subset of $C$ is a vertex cover.  It is easy to see that $C = \{i_1,\ldots,i_r\}$ is a minimal vertex cover of $G$ if and only if the prime ideal $P_C = (x_{i_1},\ldots,x_{i_r}) \subseteq S$ is a minimal prime of $I_S(G)$ \cite[9.1.4]{Herzog:Hibi:monomial:ideals}.  The \term{vertex cover number} $\beta(G)$ is the smallest number of vertices in any vertex cover of $G$.  Hence, $\hgt I_S(G) = \beta(G)$.  As a consequence of \cite[3.1.21]{West:graph:theory}, we have the following corollary.

\begin{cor}
For any graph $G$, we have $\depth_E W(G,1) = \hgt I_S(G)$. \qed
\end{cor}

Since $W(G,1)$ is Cohen-Macaulay \cite[2.2]{Villarreal90}, we also recover a special case of a purely commutative result of Biermann and Van Tuyl \cite[4.7]{balanced:simplicial:complexes} which gives the regularity of whiskered graphs.

\begin{cor}
For any graph $G$, we have $\reg W(G, 1) = \dim S/I_S(G)$.
\end{cor}

\begin{proof}
Let $n$ be the number of vertices of $G$, and let $S'$ denote the polynomial ring of $W(G, 1)$ in $2n$ variables.  Applying Theorem \ref{depth:via:maximal:shifts}, the previous corollary,  and the fact that $W(G, 1)$ is Cohen-Macaulay gives
\[ 
\hgt I_S(G) = \depth_E W(G, 1) = 2n - \pd_{S'} W(G, 1) - \reg W(G, 1). \]
The set of vertices of $G$ is a minimal vertex cover, and since $W(G, 1)$ is Cohen-Macaulay, all minimal vertex covers have the same size so that $\pd_{S'} W(G, 1) = \hgt I_{S'}(W(G, 1)) = n$.  And so, it follows that $\reg W(G, 1) = n - \hgt I_S(G) = \dim S/I_S(G)$.
\end{proof}

\begin{example}
As a consequence of Theorem \ref{depth:one:whisker}, we can compute the depth of every Cohen-Macaulay tree $G$, as every such tree is of the form $W(G_0, 1)$ for some tree $G_0$ \cite[2.4, 2.5]{Villarreal90}.  For example, the tree $G$ below is Cohen-Macaulay since it is the whiskering of the tree $G_0$ on the vertices $v_1, \dots v_{10}$, but it is not a spider.
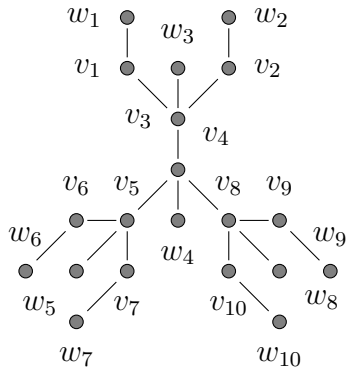
\begin{center}
\labeledgraph[scale = 0.9]{
    -1/3/0/w_1/180,
    -1/2/1/v_1/180,
    1/3/2/w_2/0,
    1/2/3/v_2/0,
    0/2/4/w_3/90,
    0/1/5/v_3/180,
    0/0/6/v_4/45,
    0/-1/7/w_4/-90,
    -1/-1/8/v_5/90,
    -2/-2/9/w_5/-120,
    -1/-2/10/v_7/-90,
    -2/-3/11/w_7/-90,
    -2/-1/12/v_6/90,
    -3/-2/13/w_6/90,
    1/-1/14/v_8/90,
    2/-2/15/w_8/-30,
    1/-2/16/v_{10}/-90,
    2/-3/17/w_{10}/-90,
    2/-1/18/v_9/90,
    3/-2/19/w_9/90
    }{
    0/1,
    1/5,
    2/3,
    3/5,
    4/5,
    5/6,
    6/7,
    6/8,
    8/9,
    8/10,
    10/11,
    8/12,
    12/13,
    6/14,
    14/15,
    14/16,
    16/17,
    14/18,
    18/19}
\end{center}
We have $\depth_E G = \beta(G_0) = 3$ since $\{v_3, v_5, v_8\}$ is a vertex cover of $G_0$ of minimum size.
\end{example}

\begin{cor} \label{bunch:of:depths} \label{depth:whisker:path} \label{depth:whisker:cycle}
 Let $n \geq 1$. Then: 
 \begin{enumerate}[label = \textnormal{(\alph*)}]
     \item $\depth_E W(P_n,1) = \floor{\frac{n}{2}}$ \\[-1 ex]
     \item $\depth_E W(C_n,1) = \ceiling{\frac{n}{2}}$ for $n \geq 3$ \\[-1 ex]
     \item $\depth_E W(K_n,1) = n-1$
 \end{enumerate} 
\end{cor}

\begin{proof}
 (a) If $P_n$ has edges $v_iv_{i+1}$ for each $i < n$, then every independent set of $P_n$ is of the of the form $\{v_{i_1}, \dots, v_{i_t}\}$ where $1 \leq i_1 \leq \cdots \leq v_{i_t} \leq n$ and $i_{j+1} \geq i_j + 2$ for all $j < t$.  From this, it easily follows that $t \leq \floor{\frac{n+1}{2}} = \ceiling{\frac{n}{2}}$. In particular, we see that $\{v_1, v_3, \dots, v_{2t -1}\}$ is independent for $t = \ceiling{\frac{n}{2}}$ so that $\alpha(P_n)=\ceiling{\frac{n}{2}}$ and $\depth_E W(P_n,1) = n - \ceiling{\frac{n}{2}} = \floor{\frac{n}{2}}$ by Theorem \ref{depth:one:whisker}.
 
 (b) If $A$ is an independent set of $C_n$, there is at least one vertex $v$ not in $A$ so that $A$ is independent in $C_n \setminus v = P_{n-1}$.  As we already observed that $\alpha(C_n \setminus v) \leq \alpha(C_n)$ in the proof of theorem, it follows that $\alpha(C_n) = \alpha(P_{n-1}) = \floor{\frac{n}{2}}$ so that $\depth_E W(C_n,1) = n - \floor{\frac{n}{2}} = \ceiling{\frac{n}{2}}$.
 
 (c) This follows immediately from the fact that every maximal independent set of $K_n$ consists of only one vertex.
\end{proof}

\begin{example}
The graph $W(K_4,3)$ is shown below.
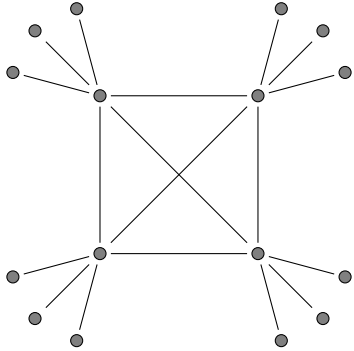
\begin{center}
\labeledgraph[scale = 0.8]{
    0/0/1//135,
    0/-3.5/2//90,
    -3.5/-3.5/3//90,
    -3.5/0/4//90,
    1.93/0.517/5//270,
    1.44/1.44/6//90,
    0.517/1.93/7//90,
    1.93/-4.017/8//90,
    1.44/-4.94/9//90,
    0.517/-5.43/10//90,
    -5.43/-4.017/11//90,
    -4.94/-4.94/12//90,
    -4.017/-5.43/13//90,
    -5.43/0.517/14//90,
    -4.94/1.44/15//90,
    -4.017/1.93/16//90}{
    1/2,
    1/3,
    1/4,
    2/3,
    2/4,
    3/4,
    1/5,
    1/6,
    1/7,
    2/8,
    2/9,
    2/10,
    3/11,
    3/12,
    3/13,
    4/14,
    4/15,
    4/16}
\end{center}
Combining the preceding corollary with Theorem \ref{multiple:whiskering:depth}, we see that  $\depth_E W(K_4,3) = (4-1)3 = 9$. The next section will show that this graph has maximal depth among all graphs with 16 vertices (and more generally, that $W(K_m, m - 1)$ has maximal depth among all graphs with $m^2$ vertices).
\end{example}

We now consider the behavior of the depth of a graph upon iteratively whiskering every vertex at least twice.  Given a graph $G$ and integers $m_1, \dots, m_k > 0$, we define the graph $W(G, m_1, \dots, m_k)$ recursively by
\[ W(G, m_1, \dots, m_k) := W(W(G, m_1, \dots, m_{k-1}), m_k) \]
When $m_i = m$ for all $i$, we also denote this graph by simply $W^k(G, m)$.  The following results show that the depth of these graphs is solely dependent on the number of vertices of the original graph and the number of whiskers added.

\begin{cor}\label{depth:repeated:whiskering} \label{depth:second:whiskering}
Let $G$ be a graph on $n$ vertices, and let $m_1, m_2, \dots, m_k > 0$ for some $k \geq 2$. Then 
\[\depth_E W(G, m_1, \dots, m_k) = n(m_1+1)(m_2+1)\cdots (m_{k-2}+1)m_k.\] 
In particular, for any $m > 0$ and $k \geq 2$, we have
\[
\depth_E W^k(G, m) = nm(m + 1)^{k-2}.
\]
\end{cor}

\begin{proof}
Suppose that $k = 2$.   By Theorem \ref{multiple:whiskering:depth}, we may assume that $m_2 = 1$.  Hence, it suffices to show that $\depth_E W(G, m_1, 1) = n$.  If $A$ is a maximal independent set of $W(G,m_1)$, we can partition $A$ into sets $A_v$ for each vertex $v$ of $G$, where $A_v$ is the set of vertices in $A$ equal to $v$ or a whisker of $v$ in $W(G, m_1)$. If $v \in A_v$, then $\abs{A_v} = 1$, and otherwise, the maximality of $A$ forces $\abs{A_v} = m_1$.  From this, it follows that the set of whiskers of $W(G, m_1)$ is the unique independent set of maximum size.  And so, Theorem \ref{depth:one:whisker} yields $\depth_E W(G,m_1,1) = n$.  This general case follows immediately by a simple induction on $k$ after applying the $k = 2$ case to the graph $G' = W(G, m_1, \dots, m_{k-2})$, which is easily seen to have $n(m_1 + 1) \cdots (m_{k-2} + 1)$ vertices.
\end{proof}

\begin{rmk}\label{whisker:remark} While the preceding corollary shows that for any two graphs $G, G'$ on $n$ vertices, $\depth_E W(G,m_1,m_2) = \depth_E W(G',m_1,m_2)$, it is not true that $\depth_E W(G,m_1) = \depth_E W(G',m_1)$.  For example, $\depth_E W(P_4,1) = 2$ while $\depth_E W(C_4,1) = 3$ by Proposition~\ref{depth:whisker:path}.  This is contrary to the case over a polynomial ring where $\depth_S W(G,1) = n$ for any graph $G$ on $n$ vertices since $W(G,1)$ is always Cohen-Macaulay \cite[2.2]{Villarreal90}.
\end{rmk}

\section{Bounds on Depth of Edge Ideals}\label{section:boundsOnDepth}

In this section, we prove some upper bounds on the depth of edge ideals.  The first main result is a tight upper bounds on the depth of arbitrary graphs.  Then we prove a more refined bound for bipartite graphs.

Let $G$ be a graph.  A subset of pairwise disjoint edges is a \term{matching}; if the edges in a matching form an induced subgraph of $G$, it is an \term{induced matching}.  The maximum size of an induced matching of $G$ is the \term{induced matching number}, denoted $\im(G)$.  
The maximal size of a minimal vertex cover is denoted $\tau_{\max}(G)$.  In particular, for any minimal vertex cover $C$ of a graph $G$, we have
\[\abs{C} \le \tau_{\max}(G) = \mathrm{bight} I_S(G) \le \pd_S S/I_S(G). \]

Next, we give a general upper bound on the depth of exterior and symmetric algebra edge ideals, the first of which recovers \cite[4.2]{HH21}.

\begin{thm}\label{general:bound} 
Let $G$ be a graph on $n$ vertices, none of which are isolated.  Then
\begin{center}
\begin{minipage}{0.4\textwidth}
\begin{align*}
    \depth_S G &\le n + 2 - 2 \sqrt{n},\\
    \depth_E G &\le n + 1 - 2 \sqrt{n},
\end{align*}
\end{minipage}
\begin{minipage}{0.4\textwidth}
\begin{align*}
    \pd_S G &\ge 2 \sqrt{n} - 2, \\
    \cx_E G &\ge 2 \sqrt{n} - 1.
\end{align*}
\end{minipage}
\vspace{1 ex}
\end{center}
Moreover, when $n$ is a square, both sets of bounds are tight.
\end{thm}

\begin{proof}
By \cite[4.7]{DS13}, $\pd_S G \ge n - \mathrm{im}(G)$.  That $\mathrm{im}(G) \le n + 2 - 2\sqrt{n}$ was noted first in \cite{Favaron88} and follows from \cite[Theorem 6]{BC79}.  Combining these yields $\pd_S G \ge 2\sqrt{n} - 2$.  By the Auslander-Buchsbaum Formula \cite[15.3]{graded:syzygies}, $\depth_S G \le n + 2 - 2\sqrt{n}$. The latter two inequalities follow from  Corollary~\ref{depth:compare} and Theorem~\ref{auslander:buchsbaum}.  That both bounds are tight follows from the proof of Corollary~\ref{tight:corollary}.
\end{proof}

\begin{cor} \label{tight:corollary}
Suppose $G$ is a graph on $n = m^2$ vertices with no isolated vertices. Then $\depth_E G = n + 1 - 2\sqrt{n}$ if and only if $G = W(K_m,m-1)$.
\end{cor}

\begin{proof} Suppose $\depth_E G = n + 1 - 2\sqrt{n} = m^2 - 2m + 1$.  By Corollary~\ref{depth:compare} and Theorem \ref{general:bound}, $\depth_S G = m^2 - 2m+2$, and so, $\pd_S G = 2m-2$. By \cite[3.1]{HH21}, $\tau_{\max}(G) = 2m - 2$ and hence by \cite[5.4]{HH21}, $G$ is one of the following graphs:
\begin{enumerate}
\item $P_2 \sqcup P_2$
\item $C_4$
\item $W(K_m,m-1)$.
\end{enumerate}
By Theorem~\ref{depth:cycles}, $\depth_E C_4 = 0 \neq 1$, so  we may disregard $C_4$.  That $\depth_E P_2 \sqcup P_2 = 0 \neq 1$ follows from Proposition \ref{disjoint:union} and Corollary~\ref{depth:path:graphs}.  Finally, $\depth_E W(K_m, m-1) = (m-1)^2$ by Theorem \ref{multiple:whiskering:depth} and Corollary \ref{bunch:of:depths}, completing the proof.
\end{proof}

Next, we focus on the depth of bipartite graphs where stronger statements are possible.  We first compute the depth of Ferrers graphs.
Let $\lambda_1 > \lambda_2 > \cdots > \lambda_m > 0$ be a sequence of positive integers.  The associated \term{Ferrers graph} is the bipartite graph on vertex set $\{v_1,\ldots,v_{m}\} \sqcup \{w_1,\ldots,w_{\lambda_1}\}$ and edge set $\{\{v_i,w_j\}\mid 1 \le i \le m \text{ and } 1 \le j \le \lambda_i\}$.  Ferrers graphs are exactly those bipartite graphs with the property that, after potentially reordering the vertices, if $\{v_i,w_j\}$ is an edge then so is $\{v_k,w_l\}$ for all $1 \le k \le i$ and $1 \le l \le j$.  Corso and Nagel characterized Ferrers graphs as exactly the bipartite graphs whose symmetric edge ideals have linear free resolutions in \cite[4.2]{CN09}.  Using this, we compute the depths of edge ideals of Ferrers graphs.

\begin{prop}\label{maximal:exterior:depth:biparite:case} Suppose $G$ is a Ferrers graph  associated to a sequence of integers $\lambda_1 > \lambda_2 >\cdots >\lambda_m > 0$.  Then 
\[\depth_S G = m + \min_j\{\lambda_1 - \lambda_j - j + 1\}\]
and
\[\depth_E G = m + \min_j\{\lambda_1 - \lambda_j - j\}.\]
\end{prop}

\begin{proof} Let $n = \lambda_1 + m$ be the number of vertices of $G$.  It follows from \cite[2.2]{CN09} that $\pd_S G = \max_j \{\lambda_j + j - 1\}$.  By the Auslander-Buchsbaum Formula, $\depth_S G = n - \pd_S G = m + \min_j\{\lambda_1 - \lambda_j - j + 1\}$.  Since $I_S(G)$ has a linear free resolution, it follows from Corollary~\ref{depth:compare} that $\depth_E G = m + \min_j\{\lambda_1 - \lambda_j - j \}$.
\end{proof}

Finally, we give a tight upper bound on the depth of bipartite graphs.

\begin{thm}\label{bipartitebound} Let $G$ be a bipartite graph on $n$ vertices none of which are isolated.  Then
\begin{center}
\begin{minipage}{0.4\textwidth}
\begin{align*}
    \depth_S G &\le \left\lfloor \frac{n}{2} \right \rfloor,\\[0.5 ex]
    \depth_E G &\le \left\lfloor \frac{n}{2} \right \rfloor - 1,
\end{align*}
\end{minipage}
\begin{minipage}{0.4\textwidth}
\begin{align*}
    \pd_S G &\ge \ceiling{\frac{n}{2}}, \\[0.5 ex]
    \cx_E G &\ge \ceiling{\frac{n}{2}} + 1.
\end{align*}
\end{minipage}
\vspace{1 ex}
\end{center}
Moreover, both sets of bounds are tight.
\end{thm}

\begin{proof} Let $G$ be bipartite on vertex set $V = A \sqcup B$.  Then each of $A$ and $B$ are minimal vertex covers of $G$.  So $\ceiling{\frac{n}{2}} \le \max\{\abs{A},\abs{B}\} \le \pd_S G$, and again by the Auslander-Buchsbaum Formula, $\depth_S G = n - \pd_S G \le n - \ceiling{\frac{n}{2}} = \floor{\frac{n}{2}}$.  The second two inequalities again follow from Corollary~\ref{depth:compare} and Theorem~\ref{auslander:buchsbaum}.  That both bounds are tight follows by considering the Ferrers graph associated to the sequence $\lambda_i = \ceiling{\frac{n}{2}} -i + 1$ for $1 \le i \le \floor{\frac{n}{2}}$ and applying the previous proposition.
\end{proof}

It seems  difficult to completely characterize bipartite graphs $G$ satisfying the equality $\depth_S G = \floor{\frac{n}{2}}$.  Indeed, such graphs include all Cohen-Macaulay bipartite graphs, Ferrers graphs with $\lambda_i \le m-i+1$ from the previous corollary, and at least some graphs which are neither.  There are also non-Ferrers bipartite graphs satisfying $\depth_E G = \floor{\frac{n}{2}} - 1$ such as the following.

\begin{example} Let $G$ be the following bipartite graph.
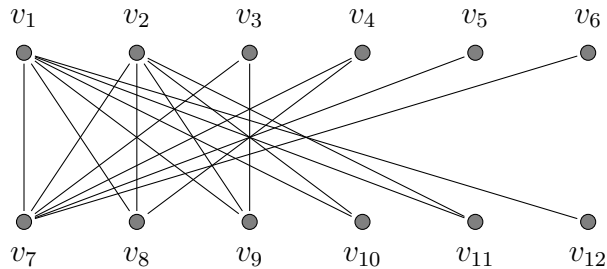
\begin{center}
\labeledgraph{
    10/0/6/v_6/90,
    8/0/5/v_5/90,
    6/0/4/v_4/90,
    4/0/3/v_3/90,
    2/0/2/v_2/90,
    0/0/1/v_1/90,
    0/-3/7/v_7/270,
    2/-3/8/v_8/270,
    4/-3/9/v_9/270,
    6/-3/10/v_{10}/270,
    8/-3/11/v_{11}/270,
    10/-3/12/v_{12}/270}{
    1/7,
    1/8,
    1/9,
    1/10,
    1/11,
    1/12,
    2/7,
    2/8,
    2/9,
    2/10,
    2/11,
    3/7,
    3/9,
    4/7,
    4/8,
    5/7,
    6/7}
    \end{center}
    The corresponding edge ideal $I_S(G)$ has the following Betti table.
    
   \begin{center}\vspace{.2cm}
    \begin{tabular}{r|ccccccc}
      &0&1&2&3&4&5&6\\
      \hline
      \text{0:}&1&\text{-}&\text{-}&\text{-}&\text{-}&\text{-}&\text{-}\\\text{1:}&\text{-}&17&50&66&47&18&3\\\text{2:}&\text{-}&\text{-}&1&3&3&1&\text{-}\\
      \end{tabular}
      \end{center}\vspace{.2cm}
      
\noindent It follows from Corollary~\ref{depth:compare} that $\depth_S G = 6 = \frac{12}{2}$ and $\depth_E G = 5  = \frac{12}{2} - 1$.  Yet, since $\reg_S G = 2$, $G$ is not a Ferrers graph, and since $(x_1,x_2,x_7,x_8,x_9)$ is a minimal prime of height $5$, $I_S(G)$ is not Cohen-Macaulay.
\end{example}

\section*{Acknowledgements}
Computations with Macaulay2 \cite{M2} were very helpful while working on this project.  Mastroeni was supported by an AMS-Simons Travel Grant.  McCullough was supported by National Science Foundation grant DMS--1900792. 

\end{spacing}

\bibliographystyle{alpha}
\bibliography{edge-ideals}

\end{document}